\documentclass[final,leqno]{siamltex704}
\usepackage{amsmath}
\usepackage{graphicx}
\usepackage[notcite,notref]{showkeys}
\usepackage{mathrsfs}
\usepackage{float}
\usepackage{amsfonts,amssymb}
\usepackage{dsfont}
\usepackage{pifont}
\usepackage{hyperref}
\usepackage{multirow}
\numberwithin{equation}{section}
\def\3bar{{|\hspace{-.02in}|\hspace{-.02in}|}}
\def\E{{\mathcal{E}}}
\def\T{{\mathcal{T}}}
\def\Q{{\mathcal{Q}}}
\def\btau{\boldsymbol{\tau}}
\def\w{\psi}
\def\bw{{\mathbf{w}}}
\def\bu{{\mathbf{u}}}
\def\bv{{\mathbf{v}}}
\def\bn{{\mathbf{n}}}
\def\be{{\mathbf{e}}}
\newtheorem{defi}{Definition}[section]

\newtheorem{algorithm}{Weak Galerkin Algorithm}
\title { An Efficient Numerical Scheme for the Biharmonic Equation by
Weak Galerkin Finite Element Methods on Polygonal or Polyhedral
Meshes}

\author{
Chunmei Wang\thanks{Jiangsu Key Laboratory for NSLSCS, School of
Mathematical Sciences, Nanjing Normal University, Nanjing 210023,
China. Second affiliation: Nanjing Normal University Taizhou
College, Taizhou 225300, China. The research of Chunmei Wang was
partially supported by {\em The Project of Graduate Education
Innovation of Jiangsu Province (CXZZ13\_0387).}} \and Junping
Wang\thanks{Division of Mathematical Sciences, National Science
Foundation, Arlington, VA 22230 (jwang@nsf.gov). The research of
Junping Wang was supported by the NSF IR/D program, while working at
National Science Foundation. However, any opinion, finding, and
conclusions or recommendations expressed in this material are those
of the author and do not necessarily reflect the views of the
National Science Foundation.}}

\begin{document}

\maketitle

\begin{abstract}
This paper presents a new and efficient numerical algorithm for the
biharmonic equation by using weak Galerkin (WG) finite element
methods. The WG finite element scheme is based on a variational form
of the biharmonic equation that is equivalent to the usual
$H^2$-semi norm. Weak partial derivatives and their approximations,
called discrete weak partial derivatives, are introduced for a class
of discontinuous functions defined on a finite element partition of
the domain consisting of general polygons or polyhedra. The discrete
weak partial derivatives serve as building blocks for the WG finite
element method. The resulting matrix from the WG method is
symmetric, positive definite, and parameter free. An error estimate
of optimal order is derived in an $H^2$-equivalent norm for the WG
finite element solutions. Error estimates in the usual $L^2$ norm
are established, yielding optimal order of convergence for all the
WG finite element algorithms except the one corresponding to the
lowest order (i.e., piecewise quadratic elements). Some numerical
experiments are presented to illustrate the efficiency and accuracy
of the numerical scheme.
\end{abstract}



\pagestyle{myheadings}

\section{Introduction}

This paper is concerned with new developments of numerical methods
for the biharmonic equation with Dirichlet and Neumann boundary
conditions. The model problem seeks an unknown function $u=u(x)$
satisfying
\begin{equation}\label{0.1}
\begin{split}
\Delta^2u&=f, \quad\text{in}\ \Omega,\\
u&=\xi, \quad\text{on}\ \partial\Omega,\\
\frac{\partial u}{\partial\textbf{n}}&=\nu, \quad  \text{on}\ \partial\Omega,\\
\end{split}
\end{equation}
where $\Omega$ is an open bounded domain in $\mathbb{R}^d$($d=2,3$)
with a Lipschitz continuous boundary $\partial\Omega$. The functions
$f$, $\xi$, and $\nu$ are given on the domain or its boundary, as
appropriate.

A variational formulation for the biharmonic problem (\ref{0.1}) is
given by seeking $u\in H^2(\Omega)$ satisfying $u|_{\partial
\Omega}=\xi$, $\frac{\partial u}{\partial \textbf{n}}|_{\partial
\Omega}=\nu$ and the following equation
\begin{equation}\label{0.2}
\sum_{i,j=1}^d(\partial^2_{ij}u,\partial^2_{ij}v)=(f,v), \quad
\forall v\in H_0^2(\Omega),
\end{equation}
where $(\cdot,\cdot)$ stands for the usual inner product in
$L^2(\Omega)$, $\partial_{ij}^2$ is the second order partial
derivative in the direction $x_i$ and $x_j$, and $H_0^2(\Omega)$ is
the subspace of the Sobolev space $H^2(\Omega)$ consisting of
functions with vanishing trace for the function itself and its
gradient.

Based on the variational form (\ref{0.2}), one may design various
conforming finite element schemes for (\ref{0.1}) by constructing
finite element spaces as subspaces of $H^2(\Omega)$. Such
$H^2$-conforming methods essentially require $C^1$-continuity for
the underlying piecewise polynomials (known as finite element
functions) on a prescribed finite element partition. The
$C^1$-continuity imposes an enormous difficulty in the construction
of the corresponding finite element functions in practical
computation. Due to the complexity in the construction of
$C^1$-continuous elements, $H^2$-conforming finite element methods
are rarely used in practice for solving the biharmonic equation.

As an alternative approach, nonconforming and discontinuous Galerkin
finite element methods have been developed for solving the
biharmonic equation over the last several decades. The Morley
element \cite{m1968} is a well-known example of nonconforming
element for the biharmonic equation by using piecewise quadratic
polynomials. Recently, a $C^0$ interior penalty method was studied
in \cite{bs2005, eghlmt2002}. In \cite{mb2007}, a hp-version
interior-penalty discontinuous Galerkin method was developed for the
biharmonic equation. To avoid the use of $C^1$-elements, mixed
methods have been developed for the biharmonic equation by reducing
the fourth order problem to a system of two second order equations
\cite{ab1985, f1978, gnp2008, m1987, mwy3818}.

Recently, weak Galerkin (WG) has emerged as a new finite element
technique for solving partial differential equations. WG method
refers to numerical techniques for partial differential equations
where differential operators are interpreted and approximated as
distributions over a set of generalized functions. The method/idea
was first introduced in \cite{wy2013} for second order elliptic
equations, and the concept was further developed in \cite{wy3655,
wy2707, mwy3655}. By design, WG uses generalized and/or
discontinuous approximating functions on general meshes to overcome
the barrier in the construction of ``smooth'' finite element
functions. In \cite{mwy0927}, a WG finite element method was
introduced and analyzed for the biharmonic equation by using
polynomials of degree $k\ge2$ on each element plus polynomials of
degree $k$ and $k-1$ for $u$ and $\frac{\partial u}{\partial\bn}$ on
the boundary of each element (i.e.,  elements of type
$P_k/P_{k}/P_{k-1}$). The WG scheme of \cite{mwy0927} is based on
the variational form of $(\Delta u, \Delta v)=(f,v)$.

In this paper, we will develop a highly flexible and robust WG
finite element method for the biharmonic equation by using an
element of type $P_k/P_{k-2}/P_{k-2}$; i.e., polynomials of degree
$k$ on each element and polynomials of degree $k-2$ on the boundary
of the element for $u$ and $\nabla u$. Our WG finite element scheme
is based on the variational form (\ref{0.2}), and has a smaller
number of unknowns than that of \cite{mwy0927} for the same order of
element. Intuitively, our WG finite element scheme for (\ref{0.1})
shall be derived by replacing the differential operator
$\partial_{ij}^2$ in (\ref{0.2}) by a discrete and weak version,
denoted by $\partial_{ij,w}^2$. In general, such a straightforward
replacement may not produce a working algorithm without including a
mechanism that enforces a certain weak continuity of the underlying
approximating functions. A weak continuity shall be realized by
introducing an appropriately defined stabilizer, denoted as
$s(\cdot,\cdot)$. Formally, our WG finite element method for
(\ref{0.1}) can be described by seeking a finite element function
$u_h$ satisfying
\begin{equation}\label{0.3}
\sum_{i,j=1}^d(\partial_{ij,w}^2u_h,\partial^2_{ij,w}v)_h+s(u_h,v)=(f,v)
\end{equation}
for all testing functions $v$. The main advantage of the present
approach as compared to \cite{mwy0927} lies in the fact that
elements of type $P_k/P_{k-2}/P_{k-2}$ are employed, which greatly
reduces the degrees of freedom and results in a smaller system to
solve. The rest of the paper is to specify all the details for
(\ref{0.3}), and justifies the rigorousness of the method by
establishing a mathematical convergence theory.

The paper is organized as follows. In Section
\ref{Section:Preliminaries}, we introduce some standard notations
for Sobolev spaces. Section \ref{Section:Wpartial} is devoted to a
discussion of weak partial derivatives and their discretizations. In
Section \ref{Section:WGFEM}, we present a weak Galerkin algorithm
for the biharmonic equation (\ref{0.1}). In Section
\ref{Section:L2projection}, we introduce some local $L^2$ projection
operators and then derive some approximation properties which are
useful in the convergence analysis. Section
\ref{Section:error-equation} will be devoted to the derivation of an
error equation for the WG finite element solution. In Section
\ref{Section:H2error}, we establish an optimal order of error
estimate for the WG finite element approximation in a
$H^2$-equivalent discrete norm. In Section \ref{Section:L2error}, we
shall derive an error estimate for the WG finite element method
approximation in the usual $L^2$-norm. Finally in Section
\ref{Section:NE}, we present some numerical results to demonstrate
the efficiency and accuracy of our WG method.

\section{Preliminaries and Notations}\label{Section:Preliminaries}
Let $D$ be any open bounded domain with Lipschitz continuous
boundary in $\mathbb{R}^d$, $d=2, 3$. We use the standard definition
for the Sobolev space $H^s(D)$ and the associated inner product
$(\cdot,\cdot)_{s,D}$, norm $\|\cdot\|_{s,D}$, and seminorm
$|\cdot|_{s,D}$ for any $s\geq 0$. For example, for any integer
$s\geq 0$, the seminorm $|\cdot|_{s,D}$ is given by
$$
|v|_{s,D}=\left(\sum_{|\alpha|=s}\int_D|\partial^\alpha
v|^2dD\right)^{\frac{1}{2}}
$$
with the usual notation
$$
\alpha=(\alpha_1,\cdots,\alpha_d),
|\alpha|=\alpha_1+\cdots+\alpha_d,
\partial^\alpha=\prod_{j=1}^d\partial_{x_j}^{\alpha_j}.
$$
The Sobolev norm $\|\cdot\|_{m,D}$ is given by
$$
\|v\|_{m,D}=\Big(\sum_{j=0}^m|v|_{j,D}^2\Big)^{\frac{1}{2}}.
$$

The space $H^0(D)$ coincides with $L^2(D)$, for which the norm and
the inner product are denoted by $\|\cdot\|_D$ and
$(\cdot,\cdot)_D$, respectively. When $D=\Omega$, we shall drop
the subscript $D$ in the norm and inner product notation.

Throughout the paper, the letter $C$ is used to denote a generic
constant independent of the mesh size and functions involved.

\section{Weak Partial Derivatives of Second
Order}\label{Section:Wpartial} For the biharmonic problem
(\ref{0.1}) with variational form (\ref{0.2}), the principle
differential operator is $\partial_{ij}^2$. Thus, we shall define
weak partial derivatives, denoted by $\partial_{ij,w}^2$, for a
class of discontinuous functions. For numerical purpose, we shall
also introduce a discrete version for the weak partial derivative
$\partial_{ij,w}^2$ in polynomial subspaces.

Let $T$ be any polygonal or polyhedral domain with boundary
$\partial T$. By a weak function on the region $T$, we mean a
function $v=\{v_0,v_b,\textbf{v}_g\}$ such that $v_0\in L^2(T)$,
$v_b\in L^{2}(\partial T)$ and $\textbf{v}_g\in [L^{2}(\partial
T)]^d$. The first and second components $v_0$ and $v_b$ can be
understood as the value of $v$ in the interior and on the boundary
of $T$. The third term, $\textbf{v}_g\in \mathbb{R}^d$ with
components $v_{gi}, i=1,\cdots,d,$ intends to represent the gradient
$\nabla v$ on the boundary of $T$. Note that $v_b$ and
$\textbf{v}_g$ may not necessarily be related to the trace of $v_0$
and $\nabla v_0$ on $\partial T$, respectively.

Denote by $W(T)$ the space of all weak functions on $T$; i.e.,
$$
W(T)=\{v=\{v_0,v_b,\textbf{v}_g\}: v_0\in L^2(T), v_b\in
L^{2}(\partial T), \textbf{v}_g\in [L^{2}(\partial T)]^d\}.
$$
Let $\langle\cdot,\cdot\rangle_{\partial T}$ be the inner product in
$L^2(\partial T)$. Define $G(T)$ by
$$
 G(T)=\{\varphi: \varphi\in H^2(T)\}.
$$

\begin{defi}\label{defition3.1} The dual of $L^2(T)$ can be identified with
itself by using the standard $L^2$ inner product as the action of
linear functionals. With a similar interpretation,  for any $v\in
W(T)$, the weak partial derivative $\partial^2_{ij}$  of
$v=\{v_0,v_b,\textbf{v}_g \}$ is defined as a linear functional
$\partial^2_{ij,w} v$ in the dual space of $G(T)$ whose action on
each $\varphi \in G(T)$ is given by
 \begin{equation}\label{2.3}
 (\partial^2_{ij,w}v,\varphi)_T=(v_0,\partial^2_{ji}\varphi)_T-
 \langle v_b n_i,\partial_j\varphi\rangle_{\partial T}+
 \langle v_{gi},\varphi n_j\rangle_{\partial T}.
 \end{equation}
Here $\textbf{n}$, with components $n_{i}\ (i=1,\cdots,d)$, is the
outward normal direction of $T$ on its boundary.
\end{defi}

Unlike the classical second order derivatives, $\partial_{ij,w}^2v$
is usually different from $\partial_{ji,w}^2v$ when $i\neq j$.

The Sobolev space $H^2(T)$ can be embedded into the space $W(T)$ by
an inclusion map $i_W: H^2(T)\rightarrow W(T)$ defined as follows
$$
i_W(\phi)=\{\phi|_T,\phi|_{\partial T},\nabla \phi|_{\partial
T}\}, \qquad \phi\in H^2(T).
$$
With the help of the inclusion map $i_W$, the Sobolev space
$H^2(T)$ can be viewed as a subspace of $W(T)$ by identifying each
$\phi\in H^2(T)$ with $i_W(\phi)$. Analogously, a weak function
$v=\{v_0,v_b,\textbf{v}_g\}\in W(T)$ is said to be in $H^2(T)$ if
it can be identified with a function $\phi\in H^2(T)$ through the
above inclusion map. It is not hard to see that
$\partial_{ij,w}^2$ is identical with  $\partial_{ij}^2$ in
$H^2(T)$; i.e., $\partial_{ij,w}^2v =\partial_{ij}^2v$ for all
functions $v\in H^2(T)$.

Next, for $i,j=1,\cdots,d$, we introduce a discrete version of
$\partial^2_{ij,w}$ by approximating $\partial^2_{ij,w}$ in a
polynomial subspace of the dual of $G(T)$. To this end, for any
non-negative integer $r\geq 0$, denote by $P_r(T)$ the set of
polynomials on $T$ with degree no more than $r$. A discrete
$\partial^2_{ij,w}\ (i,j=1,\cdots,d)$ operator, denoted by
$\partial^2_{ij,w,r,T}$, is defined as the unique polynomial
 $\partial^2_{ij,w,r,T} v\in P_r(T)$ satisfying the following equation
  \begin{equation}\label{2.4}
 (\partial^2_{ij,w,r,T}v,\varphi)_T=(v_0,\partial^2
 _{ji}\varphi)_T-\langle v_b n_i,\partial_j\varphi\rangle_{\partial T}
 +\langle v_{gi},\varphi n_j\rangle_{\partial T},\quad \forall \varphi \in
 P_r(T).
 \end{equation}

\section{ Numerical Algorithm by Weak Galerkin}\label{Section:WGFEM}

Let ${\cal T}_h$ be a partition of the domain $\Omega$ into polygons
in 2D or polyhedra in 3D. Assume that ${\cal T}_h$ is shape regular
in the sense  as defined in \cite{wy3655}. Denote by $\E_h$ the set
of all edges or flat faces in ${\cal T}_h$, and let
$\E_h^0=\E_h\setminus\partial\Omega$ be the set of all interior
edges or flat faces.

For any given integer $k\geq 2$, denote by $W_k(T)$ the discrete
weak function space given by
\begin{equation*}
W_k(T)=\big\{\{v_0,v_b,\textbf{v}_g\}: v_0\in P_k(T), v_b\in
P_{k-2}(e),\textbf{v}_g\in [P_{k-2}(e)]^d, e\subset \partial
T\big\}.
\end{equation*}
By patching $W_k(T)$ over all the elements $T\in {\cal T}_h$ through
a common value on the interface $\E_h^0$, we arrive at a weak finite
element space $V_h$ defined as follows
$$
V_h=\big\{\{v_0,v_b,\textbf{v}_g\}:\{v_0,v_b,\textbf{v}_g\}|_T\in
W_k(T), \forall T\in {\cal T}_h\big\}.
$$
Denote by $V_h^0$ the subspace of $V_h$ with vanishing trace; i.e.,
$$
V_h^0=\{\{v_0,v_b,\textbf{v}_g\}\in
V_h,v_b|_e=0,\textbf{v}_g|_e=\textbf{0}, e\subset \partial T\cap
\partial\Omega\}.
$$

Intuitively, the finite element functions in $V_h$ are piecewise
polynomials of degree $k\ge 2$. The extra value on the boundary of
each element is approximated by polynomials of degree $k-2$ for the
function itself and its gradient. For such functions, we may compute
the weak second order derivative $\partial^2_{ij,w} v$ by using the
formula (\ref{2.3}). For computational purpose, this weak partial
derivative $\partial^2_{ij,w} v$ has to be approximated by using
polynomials, preferably one with degree $k-2$. Denote by
$\partial^2_{ij, w,k-2}$ the discrete weak partial derivative
computed by using (\ref{2.4}) on each element $T$ for $k\geq 2$;
i.e.,
$$
(\partial^2_{ij, w,k-2} v)|_T=\partial^2_{ij,w,k-2,T}(v|_T), \qquad
v\in V_h.
$$
For simplicity of notation and without confusion, we shall drop the
subscript $k-2$ in the notation $\partial^2_{ij, w,k-2}$. We also
introduce the following notation
$$
(\partial^2_{w}u,\partial^2_{w}v)_h=\sum_{T\in{\cal
T}_h}\sum_{i,j=1}^d (\partial^2_{ij,w}u,\partial^2_{ij,w}v)_T,\quad
\forall u, v\in V_h.
$$

For each element $T$, denote by $Q_0$ the $L^2$ projection onto
$P_k(T)$, $k\geq 2$. For each edge or face $e\subset\partial T$,
denote by $Q_b$ the $L^2$ projection onto $P_{k-2}(e)$ or
$[P_{k-2}(e)]^d$, as appropriate. For any $w\in H^2(\Omega)$, we
define a projection $Q_h w$ into the weak finite element space $V_h$
such that on each element $T$,
$$
Q_hu=\{Q_0u,Q_bu,Q_b(\nabla u)\}.
$$

For any $w=\{w_0,w_b,\textbf{w}_g\}$ and
$v=\{v_0,v_b,\textbf{v}_g\}$ in $V_h$, we introduce a bilinear form
as follows
\begin{equation*}
\begin{split}
s(w,v)=&\sum_{T\in {\cal T}_h}  h_T^{-1}\langle Q_b(\nabla
w_0)-\textbf{w}_g, Q_b(\nabla v_0)-
\textbf{v}_g\rangle_{\partial T}\\
&+\sum_{T\in {\cal T}_h} h_T^{-3}\langle Q_b w_0-w_b, Q_b
v_0-v_b\rangle_{\partial T}.
\end{split}
\end{equation*}

The following is a precise statement of the WG finite element scheme
for the biharmonic equation (\ref{0.1}) based on the variational
formulation (\ref{0.2}).

\begin{algorithm} Find $u_h=\{u_0,u_b,\textbf{u}_g\}\in V_h$
satisfying $u_b=Q_b\xi$, $\textbf{u}_g\cdot \textbf{n}=Q_{b}\nu$,
$\textbf{u}_g\cdot \boldsymbol{\tau}=Q_{b}(\nabla\xi\cdot\btau)$ on
$\partial\Omega$ and the following equation:
\begin{equation}\label{2.7}
(\partial_{w}^2u_h,\partial^2_{w}v)_h+s(u_h,v)=(f,v_0), \quad
\forall v=\{v_0,v_b,\textbf{v}_g\}\in V_h^0,
\end{equation}
where $\boldsymbol{\tau}\in \mathbb{R}^d$ is the tangential
direction to the edges/faces on the boundary $\partial\Omega$.
\end{algorithm}

The following is a useful observation concerning the finite element
space $V_h^0$.

\begin{lemma}\label{Lemma4.1} For any $v\in V_h^0$, define $\3barv\3bar$ by
\begin{equation}\label{3barnorm}
\3barv\3bar^2= (\partial^2_{ w}v,\partial^2_{ w}v)_h+s(v,v).
\end{equation}
Then, $\3bar\cdot\3bar$ is a norm in the linear space $V_h^0$.
\end{lemma}

\begin{proof} We shall only verify the positivity property for
$\3bar\cdot\3bar$. To this end, assume that $\3barv\3bar=0$ for some
$v\in V_h^0$. It follows from (\ref{3barnorm}) that
$\partial^2_{ij,w}v=0$ on $T$, $Q_b(\nabla v_0)=\textbf{v}_g$ and
$Q_bv_0=v_b$ on $\partial T$. We claim that $\partial^2_{ij}v_0=0$
on each element $T$. To this end, for any $\varphi \in P_{k-2}(T)$,
we use $\partial^2_{ij,w}v=0$ and the identity (\ref{A.002}) to
obtain
\begin{equation*}
\begin{split}
0=&(\partial^2_{ij,w}v,\varphi)_T\\
 =&(\partial^2_{ij}v_0, \varphi)_T+
 \langle v_{gi}-Q_{b}(\partial_i v_0),\varphi\cdot n_j\rangle_{\partial T}
 +\langle Q_bv_0-v_b,\partial_j \varphi \cdot n_i\rangle_{\partial T}\\
 =&(\varphi,\partial^2_{ij}v_0)_T,
 \end{split}
\end{equation*}
which implies that $\partial^2_{ij}v_0=0$ for $i,j=1,\ldots, d$ on
each element $T$. Thus, $v_0$ is a linear function on $T$ and
$\nabla v_0$ is a constant on each element. The condition
$Q_b(\nabla v_0)=\textbf{v}_g$ on $\partial T$ implies that $\nabla
v_0=\textbf{v}_g$ on $\partial T$. Thus, $\nabla v_0$ is continuous
over the whole domain $\Omega$. The fact that $\textbf{v}_g =0$ on
$\partial\Omega$ leads to $\nabla v_0=0$ in $\Omega$ and
$\textbf{v}_g=0$ on each edge/face. Thus, $v_0$ is a constant on
each element $T$. This, together with the fact that $Q_bv_0=v_b$ on
$\partial T$, indicates that $v_0$ is continuous over the whole
domain $\Omega$. It follows from $v_b=0$ on $\partial\Omega$ that
$v_0=0$ everywhere in the domain $\Omega$. Furthermore,
$v_b=Q_b(v_0)=0$ on each edge/face. This completes the proof of the
lemma.
\end{proof}

\begin{lemma}\label{Lemma4.2} The Weak Galerkin Algorithm (\ref{2.7}) has a
unique solution.
\end{lemma}

\begin{proof} Let $u_h^{(1)}$ and $u_h^{(2)}$ be two different
solutions of  the Weak Galerkin Algorithm (\ref{2.7}). It is clear
that the difference  $e_h=u_h^{(1)}-u_h^{(2)}$ is a finite element
function in $V_h^0$ satisfying
\begin{equation}\label{2.10}
 (\partial^2_{ w}e_h,\partial^2_{ w}v)_h+s(e_h,v)=0, \quad \forall v \in V_h^0.
\end{equation}
By setting  $v=e_h$ in  (\ref{2.10}), we obtain
$$
 (\partial^2_{w}e_h,\partial^2_{w}e_h)_h+s(e_h,e_h)=0.
$$
From Lemma 4.1, we get $e_h\equiv 0$, i.e.,
$u_h^{(1)}=u_h^{(2)}$.
\end{proof}

The rest of the paper will provide a mathematical and computational
justification for the WG finite element method (\ref{2.7}).

\section{$L^2$ Projections and Their
Properties}\label{Section:L2projection}
The goal of this section is to establish some technical results for
the $L^2$ projections. These results are valuable in the error
analysis for the WG finite element method.

\begin{lemma}\label{Lemma5.1} On each element $T\in {\cal T}_h$, let ${\cal
Q}_h$ be the local $L^2$ projection onto $P_{k-2}(T)$. Then, the
$L^2$ projections $Q_h$ and ${\cal Q}_h$ satisfy the following
commutative property:
\begin{equation}\label{l}
\partial^2_{ij,w}(Q_h w)={\cal Q}_h(\partial^2_{ij} w),\qquad \forall
i,j=1,\ldots,d,
\end{equation}
 for all $w\in H^2(T)$.
\end{lemma}

\begin{proof} For $\varphi\in P_{k-2}(T)$ and $w\in H^2(T)$, from the definition of $\partial^2_{ij,w}$
and the usual integration by parts, we have
\begin{equation*}
\begin{split}
(\partial^2_{ij,w}(Q_h w),\varphi)_T&=(Q_0
w,\partial^2_{ji}\varphi)_T -\langle Q_b w,\partial_j \varphi\cdot
n_i\rangle_{\partial T}+
\langle Q_{b}(\partial_i w)\cdot n_j,\varphi\rangle_{\partial T}\\
&=(w,\partial^2_{ji}\varphi)_T-\langle w,\partial_j
 \varphi\cdot n_i\rangle_{\partial T}+
 \langle \partial_i w\cdot n_j,\varphi\rangle_{\partial T}\\
&=(\partial^2_{ij}w,\varphi)_T\\
&=({\cal Q}_h\partial^2_{ij}w,\varphi)_T, \quad \forall
i,j=1,\cdots,d,
\end{split}
\end{equation*}
which completes the proof.
\end{proof}

The commutative property (\ref{l}) indicates that the discrete weak
partial derivative of the $L^2$ projection of a smooth function is a
good approximation of the classical partial derivative of the same
function. This is a nice and useful property of the discrete weak
partial differential operator $\partial^2_{ij,w}$ in application to
algorithm design and analysis.

The following lemma provides some approximation properties for the
projection operators $Q_h$ and ${\cal Q}_h$.

\begin{lemma}\label{Lemma5.2}\cite{wy3655, mwy0927}  Let ${\cal T}_h$ be a
finite element partition of  $\Omega$ satisfying the shape
regularity assumption as defined in \cite{wy3655}. Then, for any
$0\leq s\leq 2$ and $1\leq m\leq k$, we have
\begin{equation}\label{3.2}
\sum_{T\in {\cal T}_h}h_T^{2s}\|u-Q_0u\|^2_{s,T}\leq Ch^{2(m+1)}\|u\|_{m+1}^2,
\end{equation}
\begin{equation}\label{3.3}
\sum_{T\in {\cal
T}_h}\sum_{i,j=1}^dh_T^{2s}\|\partial^2_{ij}u-{\cal
Q}_h\partial^2_{ij}u\|^2_{s,T}\leq Ch^{2(m-1)}\|u\|_{m+1}^2.
\end{equation}
\end{lemma}

Using Lemma \ref{Lemma5.2} we can prove the following result.

\begin{lemma}\label{Lemma5.3} Let $1\leq m\leq k$ and $u\in
H^{\max\{m+1,4\}}(\Omega)$. There exists a constant $C$ such that
the following estimates hold true:
\begin{equation}\label{3.5}
\Big(\sum_{T\in {\cal
T}_h}\sum_{i,j=1}^dh_T\|\partial^2_{ij}u-{\cal
Q}_h\partial^2_{ij}u\|_{\partial T}^2\Big)^{\frac{1}{2}}\leq
Ch^{m-1}\|u\|_{m+1},
\end{equation}
\begin{equation}\label{3.6}
\Big(\sum_{T\in {\cal
T}_h}\sum_{i,j=1}^dh_T^3\|\partial_j(\partial^2
_{ij}u-{\cal Q}_h\partial^2_{ij}u)\|_{\partial T}^2\Big)^{\frac{1}{2}}\\
\leq Ch^{m-1}(\|u\|_{m+1}+h\delta_{m,2}\|u\|_4),
\end{equation}
\begin{equation}\label{3.7}
\Big(\sum_{T\in {\cal T}_h}h_T^{-1}\|Q_b(\nabla Q_0u)-Q_b(\nabla
u)\|_{\partial T}^2\Big)^{\frac{1}{2}}\leq Ch^{m-1}\|u\|_{m+1},
\end{equation}
\begin{equation}\label{3.8}
\Big(\sum_{T\in {\cal T}_h}h_T^{-3}\|Q_b
 (Q_0u)- Q_bu\|_{\partial T}^2\Big)^{\frac{1}{2}}\leq
 Ch^{m-1}\|u\|_{m+1}.
\end{equation}
Here $\delta_{i,j}$ is the usual Kronecker's delta with value $1$
when $i=j$ and value $0$ otherwise.
\end{lemma}

\begin{proof} To prove (\ref{3.5}), by the trace inequality (\ref{trace-inequality}) and the
estimate (\ref{3.3}), we get
\begin{equation*}
\begin{split}
&\sum_{T\in {\cal T}_h}\sum_{i,j=1}^dh_T\|\partial^2
_{ij}u-{\cal Q}_h\partial^2_{ij}u\|_{\partial T}^2\\
\leq & C\sum_{T\in {\cal T}_h}\sum_{i,j=1}^d\Big(\|\partial^2
_{ij}u-{\cal Q}_h\partial^2
_{ij}u\|_{T}^2+h_T^2|\partial^2_{ij}u-{\cal Q}_h\partial^2_{ij}u|_{1,T}^2\Big)\\
\leq & Ch^{2m-2}\|u\|^2_{m+1}.
\end{split}
\end{equation*}

As to (\ref{3.6}), by the trace inequality (\ref{trace-inequality})
and the estimate (\ref{3.3}), we obtain
\begin{equation*}
\begin{split}
&\sum_{T\in {\cal T}_h}\sum_{i,j=1}^dh_T^3\|\partial_j(\partial_{ij}u-{\cal Q}_h\partial_{ij}u)\|_{\partial T}^2\\
\leq & C\sum_{T\in {\cal
T}_h}\sum_{i,j=1}^d\Big(h_T^2\|\partial_j(\partial^2_{ij}u-{\cal
Q}_h\partial^2_{ij}u)\|_{T}^2
+h_T^4|\partial_{j}(\partial^2_{ij}u-{\cal Q}_h\partial^2_{ij}u)|_{1,T}^2\Big)\\
\leq & Ch^{2m-2}\big(\|u\|^2_{m+1}+h^2\delta_{m,2}\|u\|_4^2\big).
\end{split}
\end{equation*}

As to (\ref{3.7}), by  the trace inequality (\ref{trace-inequality})
and the estimate (\ref{3.2}), we have

\begin{equation*}
\begin{split}
&\sum_{T\in {\cal T}_h}h_T^{-1}\|Q_b(\nabla Q_0u)-Q_b(\nabla u)\|_{\partial T}^2\\
\leq& \sum_{T\in {\cal T}_h}h_T^{-1}\|\nabla Q_0u-\nabla u\|_{\partial T}^2\\
\leq& C\sum_{T\in {\cal T}_h}\Big(h_T^{-2}\|\nabla Q_0u-\nabla u\|_{ T}^2+|\nabla Q_0u-\nabla u|_{1,T}^2\Big)\\
\leq&  Ch^{2m-2}\|u\|^2_{m+1}.
\end{split}
\end{equation*}

Finally for (\ref{3.8}), by the trace inequality
(\ref{trace-inequality}) and the estimate (\ref{3.2}), we have
\begin{equation*}
\begin{split}
&\sum_{T\in {\cal T}_h}h_T^{-3}\|Q_b (Q_0u)- Q_bu\|_{\partial T}^2\\
\leq & \sum_{T\in {\cal T}_h}h_T^{-3}\|Q_0u-u\|_{\partial T}^2\\
\leq& C\sum_{T\in {\cal T}_h}\Big(h_T^{-4}\|Q_0u-u\|_{T}^2+h_T^{-2}\|\nabla(Q_0u-u)\|_{T}^2\Big)\\
\leq&  Ch^{2m-2}\|u\|^2_{m+1}.
\end{split}
\end{equation*}
This completes the proof of the lemma.
\end{proof}

\section{An Error Equation}\label{Section:error-equation}
Let $u$ and $u_h=\{u_0,u_b,\textbf{u}_g\} \in V_h$ be the solution
(\ref{0.1}) and (\ref{2.7}) respectively. Denote by
\begin{equation}\label{error-term}
e_h=Q_hu-u_h
\end{equation}
the error function between the $L^2$ projection of the exact
solution $u$ and its weak Galerkin finite element approximation
$u_h$. An error equation refers to some identity that the error
function $e_h$ must satisfy. The goal of this section is to derive
an error equation for $e_h$. The following is our main result.

\begin{lemma}\label{Lemma6.1} The error function
$e_h$ as defined by (\ref{error-term}) is a finite element function
in $V_h^0$ and satisfies the following equation
\begin{equation}\label{4.1}
(\partial^2_w e_h,\partial^2_w v)_h+s(e_h,v)=\phi_u(v),\qquad
\forall v\in V_h^0,
\end{equation}
where
\begin{equation}\label{phiu}
\begin{split}
\phi_u(v)=
&\sum_{T\in{\cal T}_h}\sum_{i,j=1}^d\langle\partial^2_{ij}u-{\cal Q}_h(\partial^2_{ij}u),(\partial_iv_0-v_{gi})\cdot n_j\rangle_{\partial T}\\
&-\sum_{T\in{\cal T}_h}\sum_{i,j=1}^d\langle\partial_j(\partial^2_{ij}u-{\cal Q}_h\partial^2_{ij}u)\cdot n_i,v_0-v_b\rangle_{\partial T}\\
&+s(Q_hu,v).
\end{split}
\end{equation}
\end{lemma}

\begin{proof} From Lemma \ref{Lemma5.1} we have $\partial_{ij,w}^2 Q_h
u = \Q_h (\partial_{ij}^2 u)$. Now using (\ref{A.001}) with
$\varphi=\partial_{ij,w}^2 Q_h u $ we obtain
\begin{equation*}
\begin{split}
(\partial^2_{ij}v,\partial^2_{ij,w}Q_hu)_T=&(\partial^2_{ij}v_0,{\cal
Q}_h(\partial^2_{ij}u))_T+\langle v_0-v_b,\partial_j({\cal
Q}_h(\partial^2_{ij}u))
\cdot n_i\rangle_{\partial T}\\
&-\langle(\partial_i v_0-v_{gi})\cdot n_j,{\cal Q}_h\partial^2_{ij} u\rangle_{\partial T}\\
=&(\partial^2_{ij}v_0, \partial^2_{ij}u)_T+\langle
v_0-v_b,\partial_j({\cal Q}_h(\partial^2_{ij}u))\cdot
n_i\rangle_{\partial T}\\
&-\langle(\partial_i v_0-v_{gi})\cdot n_j,{\cal Q}_h\partial^2_{ij}
u\rangle_{\partial T},
\end{split}
\end{equation*}
which implies that
\begin{equation}\label{4.2}
\begin{split}
(\partial^2_{ij}v_0, \partial^2_{ij}u)_T=&
(\partial^2_{ij,w}Q_hu,\partial^2_{ij,w}v)_T-\langle v_0-v_b,
\partial_j({\cal Q}_h(\partial^2_{ij}u))\cdot n_i\rangle_{\partial T}\\
&+\langle(\partial_i v_0-v_{gi})\cdot n_j,{\cal Q}_h\partial^2_{ij}
u\rangle_{\partial T}.
\end{split}
\end{equation}
We emphasize that (\ref{4.2}) is valid for any $v\in V_h^0$ and any
smooth function $u\in H^r(\Omega),\ r>3$. Next, it follows from the
integration by parts that
\begin{equation*}
(\partial^2_{ij}u, \partial^2_{ij}v_0)_T=
((\partial^2_{ij})^2u,v_0)_T+\langle \partial^2_{ij}u,
\partial_i v_0\cdot n_j\rangle_{\partial T}
-\langle \partial_j(\partial^2_{ij}u)\cdot n_i, v_0\rangle_{\partial
T}.
\end{equation*}
By summing over all $T$ and  using the identity that $(\triangle^2
u,v_0)=(f,v_0)$, we obtain
\begin{equation*}
\begin{split}
\sum_{T\in {\cal T}_h}\sum_{i,j=1}^d(\partial^2_{ij}u, \partial^2_{ij}v_0)_T=& (f,v_0)+\sum_{T\in {\cal T}_h}\sum_{i,j=1}^d\langle \partial^2_{ij}u,(\partial_i v_0-v_{gi})\cdot n_j\rangle_{\partial T}\\
&-\sum_{T\in {\cal T}_h}\sum_{i,j=1}^d\langle
\partial_j(\partial^2_{ij}u)\cdot n_i, v_0-v_b\rangle_{\partial
T},
\end{split}
\end{equation*}
where we have used the fact that the sum for the terms associated
with $v_{gi}\cdot n_j$ and $v_b n_i$ vanishes (note that both
$v_{gi}$ and $v_b$ vanishes on $\partial\Omega$). Combining the
above equation with (\ref{4.2}) yields
\begin{equation*}
\begin{split}
 (\partial^2_w Q_hu,\partial^2_w v)_h=& (f,v_0)+\sum_{T\in{\cal T}_h}\sum_{i,j=1}^d\langle\partial^2_{ij}u-{\cal Q}_h(\partial^2_{ij}u),(\partial_iv_0-v_{gi})\cdot n_j\rangle_{\partial T}\\
&-\sum_{T\in{\cal
T}_h}\sum_{i,j=1}^d\langle\partial_j(\partial^2_{ij}u-{\cal
Q}_h\partial^2_{ij}u)\cdot n_i,v_0-v_b\rangle_{\partial T}.
\end{split}
\end{equation*}
Adding $s(Q_hu,v)$ to both side of the above equation gives
\begin{equation}\label{4.3}
\begin{split}
 &(\partial^2_w Q_hu,\partial^2_w v)_h+s(Q_hu,v)\\
 =& (f,v_0)+\sum_{T\in{\cal T}_h}\sum_{i,j=1}^d\langle\partial^2_{ij}u-{\cal Q}_h(\partial^2_{ij}u),
 (\partial_iv_0-v_{gi})\cdot n_j\rangle_{\partial T}\\
&-\sum_{T\in{\cal
T}_h}\sum_{i,j=1}^d\langle\partial_j(\partial^2_{ij}u-{\cal
Q}_h\partial^2_{ij}u)\cdot n_i,v_0-v_b\rangle_{\partial
T}+s(Q_hu,v).
\end{split}
\end{equation}
Subtracting (\ref{2.7}) from (\ref{4.3}) yields the following error
equation
\begin{equation*}
\begin{split}
&(\partial^2_w e_h,\partial^2_w v)_h+s(e_h,v)=\sum_{T\in{\cal
T}_h}\sum_{i,j=1}^d\langle\partial^2_{ij}u-
{\cal Q}_h(\partial^2_{ij}u),(\partial_iv_0-v_{gi})\cdot n_j\rangle_{\partial T}\\
&-\sum_{T\in{\cal
T}_h}\sum_{i,j=1}^d\langle\partial_j(\partial^2_{ij}u-{\cal
Q}_h\partial^2_{ij}u)\cdot n_i,v_0-v_b\rangle_{\partial
T}+s(Q_hu,v),
\end{split}
\end{equation*}
which completes the proof.
\end{proof}

\section{Error Estimates in $H^2$}\label{Section:H2error}
The goal of this section is to derive some error estimate for the
solution of Weak Galerkin Algorithm (\ref{2.7}). From the error
equation (\ref{4.1}), it suffices to handle the term $\phi_u(v)$
defined by (\ref{phiu}).

Let $w$ be any smooth function in $\Omega$. We rewrite $\phi_w(v)$
as follows:
\begin{equation}\label{phiu-breaks}
\begin{split}
\phi_w(v)=&\sum_{T\in{\cal T}_h}\sum_{i,j=1}^d\langle\partial^2_{ij}w-{\cal Q}_h(\partial^2_{ij}w),(\partial_iv_0-v_{gi})\cdot n_j\rangle_{\partial T}\\
&-\sum_{T\in{\cal T}_h}\sum_{i,j=1}^d\langle\partial_j(\partial^2_{ij}w-{\cal Q}_h\partial^2_{ij}w)\cdot n_i, v_0-v_b\rangle_{\partial T}\\
&+ \sum_{T\in{\cal T}_h}  h_T^{-1}\langle Q_b(\nabla Q_0w)-Q_b(\nabla w), Q_b(\nabla v_0)-\textbf{v}_g \rangle_{\partial T}\\
&+\sum_{T\in{\cal T}_h} h_T^{-3}\langle Q_bQ_0w-Q_bw,
Q_bv_0-v_b\rangle_{\partial T}\\
=&I_1(w,v)+I_2(w,v)+I_3(w,v)+I_4(w,v),
\end{split}
\end{equation}
where $I_j(w,v)$ are defined accordingly. Each $I_j(w,v)$ is to be
handled as follows.

\begin{lemma}\label{lemma:IoneItwo}
Assume that $w\in H^{r+1}(\Omega), v\in V_h^0$ with $r\in [2,k]$.
Let $I_1(w,v)$ and $I_2(w,v)$ be given in (\ref{phiu-breaks}). Then,
we have
\begin{eqnarray}
|I_1(w,v)|&\le& Ch^{r-1}\|w\|_{r+1}\3bar v\3bar,\label{mmm1}\\
|I_2(w, v)|&\le& Ch^{r-1}(\|w\|_{r+1}+\delta_{k,2}\|w\|_4) \3bar
v\3bar.\label{mmm2}
\end{eqnarray}
\end{lemma}

\begin{proof}
For the term $I_1(w,v)$, we use Cauchy-Schwarz inequality, the
estimate (\ref{3.5}) with $m=r$ and Lemma \ref{Lemma6.5} to obtain
\begin{equation}\label{4.5}
\begin{split}
|I_1(w,v)|=&\Big|\sum_{T\in{\cal T}_h}\sum_{i,j=1}^d\langle\partial^2_{ij}w-{\cal Q}_h(\partial^2_{ij}w),(\partial_i v_0-v_{gi})\cdot n_j\rangle_{\partial T}\Big|\\
\leq &\Big(\sum_{T\in{\cal T}_h}\sum_{i,j=1}^d h_T\|\partial^2_{ij}w-{\cal Q}_h(\partial^2_{ij}w) \|_{\partial T}^2\Big)^{\frac{1}{2}}\cdot\\
&\Big(\sum_{T\in{\cal T}_h}\sum_{i,j=1}^d h_T^{-1}\|(\partial_i v_0-v_{gi})\cdot n_j\|_{\partial T}^2\Big)^{\frac{1}{2}}\\
\leq &Ch^{r-1}\|w\|_{r+1}\3bar v\3bar,
\end{split}
\end{equation}
which verifies (\ref{mmm1}).

As to the term $I_2(w,v)$, for the case of quadratic element $k=2$,
we use Lemma \ref{Lemma6.4} to obtain
\begin{equation}\label{before4.6}
\begin{split}
\left|\sum_{T\in {\cal
T}_h}\sum_{i,j=1}^d\langle\partial_j(\partial^2_{ij} w-{\cal
Q}_h\partial^2_{ij} w)\cdot n_i, v_0-v_b\rangle_{\partial
T}\right|\leq Ch \|w\|_4\3bar v\3bar.
\end{split}
\end{equation}
For $k\ge 3$, we use Cauchy-Schwarz inequality, the estimate
(\ref{3.6}) with $m=r$, and Lemma \ref{Lemma6.2} to obtain
\begin{equation}\label{4.6}
\begin{split}
&\Big|\sum_{T\in{\cal T}_h}\sum_{i,j=1}^d\langle\partial_j(\partial^2_{ij}w-{\cal Q}_h\partial^2_{ij}w)\cdot n_i,v_0-v_b\rangle_{\partial T}\Big|\\
\leq &\Big(\sum_{T\in{\cal T}_h}\sum_{i,j=1}^d
h_T^3\|\partial_j(\partial^2_{ij}w- {\cal Q}_h\partial^2_{ij}w)
\|_{\partial T}^2\Big)^{\frac{1}{2}}\cdot
\Big(\sum_{T\in{\cal T}_h}h_T^{-3}\|v_0-v_b\|_{\partial T}^2\Big)^{\frac{1}{2}}\\
\leq &Ch^{r-1} \ \|w\|_{r+1}\ \3bar v\3bar.
\end{split}
\end{equation}
Combining (\ref{before4.6}) with (\ref{4.6}) yields
\begin{equation}\label{before4.7}
|I_2(w,v)|\leq  Ch^{r-1}(\|w\|_{r+1}+\delta_{k,2}\|w\|_4)\ \3bar
v\3bar.
\end{equation}
This completes the proof of the lemma.
\end{proof}

\smallskip

\begin{lemma}\label{lemma:IthreeIfour}
Assume that $w\in H^{r+1}(\Omega), v\in V_h^0$ with $r\in [2,k]$.
Let $I_3(w,v)$ and $I_4(w,v)$ be given in (\ref{phiu-breaks}). Then,
we have
\begin{eqnarray}
|I_3(w,v)|+|I_4(w, v)|\le Ch^{r-1}\|w\|_{r+1} |v|_h,\label{mmm3}
\end{eqnarray}
where
\begin{equation}\label{3bar-half}
|v|_h = s(v,v)^{\frac12}.
\end{equation}
\end{lemma}

\begin{proof}
To estimate the term $I_3(w,v)$, we use Cauchy-Schwarz inequality
and the estimate (\ref{3.7}) with $m=r$ to obtain
\begin{equation}\label{4.7}
\begin{split}
|I_3(w,v)|=&\left|\sum_{T\in{\cal T}_h}h_T^{-1}\langle \nabla
Q_0w-\nabla
w,Q_b(\nabla v_0)-\textbf{v}_g \rangle_{\partial T}\right|\\
 \leq& \Big(\sum_{T\in{\cal T}_h}h_T^{-1}\|\nabla
Q_0w-\nabla w\|^2_{\partial T}\Big)^{\frac{1}{2}}\cdot
\Big(\sum_{T\in{\cal T}_h}h_T^{-1}\|Q_b(\nabla v_0)-\textbf{v}_g\|^2_{\partial T}\Big)^{\frac{1}{2}}\\
\leq& Ch^{r-1}\|w\|_{r+1} \ |v|_h.
\end{split}
\end{equation}

As to the term $I_4(w,v)$, we use Cauchy-Schwarz inequality and the
estimate (\ref{3.8}) with $m=r$ to obtain
\begin{equation}\label{4.8}
\begin{split}
|I_4(w,v)|=&\left|\sum_{T\in{\cal T}_h}h_T^{-3}\langle Q_0w-w, Q_b v_0-v_b\rangle_{\partial T}\right|\\
\leq& \Big(\sum_{T\in{\cal T}_h}h_T^{-3}\| Q_0w-w\|^2_{\partial T}\Big)^{\frac{1}{2}}\Big(\sum_{T\in{\cal T}_h}h_T^{-3}\|Q_b v_0-v_b\|^2_{\partial T}\Big)^{\frac{1}{2}}\\
\leq& Ch^{r-1}\|w\|_{r+1}\ |v|_h.
\end{split}
\end{equation}
This completes the proof.
\end{proof}

The following result is an estimate for the error function $e_h$ in
the trip-bar norm which is essentially an $H^2$-equivalent norm in
$V_h^0$.

\begin{theorem}\label{Theorem6.6} Let $u_h\in V_h$ be the weak Galerkin finite
element solution arising from (\ref{2.7}) with finite elements of
order $k\geq 2$. Assume that the exact solution $u$ of (\ref{0.1})
is sufficiently regular such that $u\in H^{\max\{k+1,4\}}(\Omega)$.
Then, there exists a constant $C$ such that
\begin{equation}\label{4}
\3baru_h-Q_hu\3bar\leq
Ch^{k-1}\Big(\|u\|_{k+1}+\delta_{k,2}\|u\|_4\Big).
\end{equation}
In other words, we have an optimal order of convergence in the $H^2$
norm.
\end{theorem}

\begin{proof} By letting  $v=e_h$ in the error equation (\ref{4.1}),
we obtain the following identity
\begin{equation}\label{4.4}
\begin{split}
\3bar e_h\3bar^2=&\phi_u(e_h)\\
=&I_1(u,e_h)+I_2(u,e_h)+I_3(u,e_h)+I_4(u,e_h),
\end{split}
\end{equation}
Using the estimates (\ref{mmm1}), (\ref{mmm2}), and (\ref{mmm3})
with $w=u$ and $v=e_h$ we arrive at
$$
\3bar e_h\3bar^2 \leq
Ch^{k-1}\Big(\|u\|_{k+1}+\delta_{k,2}\|u\|_4\Big)\3bar e_h\3bar,
$$
which implies the desired error estimate (\ref{4}).
\end{proof}

\section{Error Estimates in $L^2$}\label{Section:L2error}
This section shall establish an estimate for the first component of
the error function $e_h$ in the standard $L^2$ norm. To this end, we
consider the following dual problem:
\begin{equation}\label{5.1}
\begin{split}
\Delta^2\w&=e_0 \qquad \text{in}\ \Omega,\\
\w&=0 \qquad \text{on}\ \partial\Omega,\\
\frac{\partial \w}{\partial\textbf{n}}&=0  \qquad \text{on}\
\partial\Omega.
\end{split}
\end{equation}
Assume the above dual problem has the following regularity estimate
\begin{equation}\label{5.2}
\|\w\|_4\leq C\|e_0\|.
\end{equation}

\begin{theorem}\label{Theorem7.3} Let $u_h\in V_h$ be the solution of
the Weak Galerkin Algorithm (\ref{2.7}) with finite elements of
order $k\geq 2$. Let $t_0=\min\{k,3\}$. Assume that the exact
solution of (\ref{0.1}) is sufficiently regular so that $u\in
H^{4}(\Omega)$ for $k=2$ and $u\in H^{k+1}(\Omega)$ otherwise, and
the dual problem (\ref{5.1}) has the $H^4$ regularity. Then, there
exists a constant $C$ such that
\begin{equation}\label{5.3}
\|Q_0u-u_0\|\leq
Ch^{k+t_0-2}\Big(\|u\|_{k+1}+\delta_{k,2}\|u\|_{4}\Big).
\end{equation}
In other words, we have a sub-optimal order of convergence for $k=2$
and optimal order of convergence for $k\geq 3$.
\end{theorem}

\begin{proof} By testing (\ref{5.1}) against the error function
$e_0$ on each element and using the integration by parts, we obtain
\begin{equation*}
\begin{split}
\|e_0\|^2=&(\Delta^2 \w,e_0)\\
=&\sum_{T\in {\cal
T}_h}\sum_{i,j=1}^d\Big\{(\partial^2_{ij}\w,\partial^2_{ij}e_0)_T-\langle\partial^2_{ij}\w,
\partial_ie_0\cdot n_j\rangle_{\partial T}+\langle \partial_j(\partial^2_{ij}\w)\cdot n_i,e_0\rangle_{\partial T}\Big\}\\
=&\sum_{T\in {\cal
T}_h}\sum_{i,j=1}^d\Big\{(\partial^2_{ij}\w,\partial^2_{ij}e_0)_T-\langle\partial^2_{ij}\w,
(\partial_ie_0-e_{gi})\cdot n_j\rangle_{\partial T}\\
&+\langle\partial_j(\partial^2_{ij}\w)\cdot n_i,
e_0-e_b\rangle_{\partial T}\Big\},
\end{split}
\end{equation*}
where the added terms associated with $e_b$ and $e_{gi}$ vanish due
to the cancelation for interior edges and the fact that $e_b$ and
$e_{gi}$ have zero value on $\partial\Omega$. Using (\ref{4.2}) with
$\w$ and $e_h$ in the place of $u$ and $v_0$ respectively, we arrive
at
\begin{equation}\label{09-100}
\begin{split}
\|e_0\|^2 =&(\partial^2_w Q_h\w, \partial^2_w e_h)_h
+\sum_{T\in {\cal T}_h}\sum_{i,j=1}^d\Big\{\langle\partial_j(\partial^2_{ij}\w-{\cal Q}_h(\partial^2_{ij}\w))\cdot n_i,  e_0-e_b\rangle_{\partial T}\\
&-\langle\partial^2_{ij}\w-{\cal
Q}_h\partial^2_{ij}\w,(\partial_ie_0-e_{gi})\cdot
n_j\rangle_{\partial T}\Big\}\\
=&(\partial^2_w Q_h\w, \partial^2_w
e_h)_h-\phi_\w(e_h)+s(Q_h\w,e_h).
\end{split}
\end{equation}

Next, it follows from the error equation (\ref{4.1}) that
\begin{equation}\label{09-101}
(\partial^2_w Q_h\w,\partial^2_w e_h)_h=\phi_u(Q_h\w)-s(e_h,Q_h\w).
\end{equation}
Substituting (\ref{09-101}) into (\ref{09-100}) yields
\begin{equation}\label{5.4}
\|e_0\|^2=\phi_u(Q_h\w) - \phi_\w(e_h).
\end{equation}
The term $\phi_\w(e_h)$ can be handled by using Lemma
\ref{lemma:IoneItwo} and Lemma \ref{lemma:IthreeIfour} with
$r=t_0=\min\{k,3\}$ as follows:
\begin{equation}\label{phi-psi-eh}
\begin{split}
|\phi_\w(e_h)|&\leq C h^{t_0-1} (\|\psi\|_{t_0+1}+ h\|\psi\|_4)\3bar
e_h\3bar\\
&\leq C h^{t_0-1} \|\psi\|_4 \3bar e_h\3bar\\
&\leq C h^{t_0-1} \|e_0\| \3bar e_h\3bar,
\end{split}
\end{equation}
where we have used the regularity assumption (\ref{5.2}) in the last
inequality.

It remains to deal with the term $\phi_u(Q_h\w)$ in (\ref{5.4}).
Note that from (\ref{phiu-breaks}) we have
\begin{equation}\label{raining.100}
\phi_u(Q_h\w)= \sum_{j=1}^4 I_j(u,Q_h\w).
\end{equation}

$I_3(u,Q_h\w)$ and $I_4(u,Q_h\w)$ can be handled by using Lemma
\ref{lemma:IthreeIfour} with $r=k$ as follows:
\begin{equation}\label{L2-IthreeIfour}
|I_3(u,Q_h\w)|+|I_3(u,Q_h\w)|\leq Ch^{k-1}\|u\|_{k+1} |Q_h\w|_h.
\end{equation}
From the definition (\ref{3bar-half}) we have
\begin{equation*}
\begin{split}
|Q_h\w|_h^2 &= \sum_{T\in\T_h}
\left(h_T^{-3}\|Q_b(Q_0\w)-Q_b\w\|_{\partial T}^2 +
h_T^{-1}\|Q_b(\nabla Q_0\w)-Q_b\nabla\w\|_{\partial T}^2 \right)\\
&\leq \sum_{T\in\T_h} \left(h_T^{-3}\|Q_0\w-\w\|_{\partial T}^2 +
h_T^{-1}\|\nabla (Q_0\w)-\nabla\w\|_{\partial T}^2 \right)
\end{split}
\end{equation*}
Thus, it follows from the trace inequality (\ref{trace-inequality})
and the error estimate for the projection operator $Q_0$ that
\begin{equation}
|Q_h\w|_h\leq C h^{t_0-1}\|\psi\|_{t_0+1}\leq C
h^{t_0-1}\|\psi\|_{4} \leq C h^{t_0-1}\|e_0\|.
\end{equation}
Substituting the above estimate into (\ref{L2-IthreeIfour}) yields
\begin{equation}\label{L2-IthreeIfour-new}
|I_3(u,Q_h\w)|+|I_3(u,Q_h\w)|\leq Ch^{k+t_0-2}\|u\|_{k+1} \|e_0\|.
\end{equation}

The estimate for $I_1(u,Q_h\w)$ and $I_2(u,Q_h\w)$ shall explore the
special property of the ``test" function $Q_h\w$. To this end, using
the orthogonality property of $Q_b$ and the fact that $\w=Q_b\w=0$
on $\partial\Omega$ we obtain
\begin{equation*}
\begin{split}
&\sum_{T\in{\cal T}_h}\sum_{i,j=1}^d\langle
\partial_j(\partial^2_{ij}u-{\cal Q}_h\partial^2_{ij}u)\cdot n_i,
\w-Q_b\w\rangle_{\partial T}\\
=& \sum_{T\in{\cal T}_h}\sum_{i,j=1}^d\langle
\partial_j\partial^2_{ij}u\cdot n_i,\w-Q_b\w\rangle_{\partial T}
=0.
\end{split}
\end{equation*}
Thus,
\begin{equation*}
\begin{split}
I_2(u,Q_h\w)=&-\sum_{T\in{\cal T}_h}\sum_{i,j=1}^d\langle
\partial_j(\partial^2_{ij}u-{\cal Q}_h\partial^2_{ij}u)\cdot n_i,
Q_0\w-Q_b\w\rangle_{\partial T}\\
=&-\sum_{T\in{\cal T}_h}\sum_{i,j=1}^d\langle
\partial_j(\partial^2_{ij}u-{\cal Q}_h\partial^2_{ij}u)\cdot n_i,
Q_0\w-\w\rangle_{\partial T}.
\end{split}
\end{equation*}
Using the Cauchy-Schwarz inequality and the standard error estimate
for $L^2$ projections we arrive at
\begin{equation}\label{raining.200}
\begin{split}
|I_2(u,Q_h\w)|\leq& \sum_{T\in{\cal T}_h}\sum_{i,j=1}^d
\|\partial_j(\partial^2_{ij}u-{\cal
Q}_h\partial^2_{ij}u)\|_{\partial T} \|Q_0\w-\w\|_{\partial T}\\
\leq & C h^{k+t_0-2} (\|u\|_{k+1} + \delta_{k,2}\|u\|_4)\
\|\w\|_{t_0+1} \\
\leq & C h^{k+t_0-2}(
\|u\|_{k+1} + \delta_{k,2}\|u\|_4)\ \|\w\|_{4}\\
\leq & C h^{k+t_0-2} (\|u\|_{k+1}+\delta_{k,2}\|u\|_4)\|e_0\|.
\end{split}
\end{equation}
A similar argument can be employed to deal with the term $I_1(u,
Q_h\w)$, yielding
\begin{equation}\label{raining.300}
\begin{split}
|I_1(u,Q_h\w)|\leq C h^{k+t_0-2} \|u\|_{k+1} \|e_0\|.
\end{split}
\end{equation}
Substituting (\ref{L2-IthreeIfour-new}), (\ref{raining.200}), and
(\ref{raining.300}) into (\ref{raining.100}) we arrive at
\begin{equation}\label{raining.400}
|\phi_u(Q_h\w)|\leq C h^{k+t_0-2} (\|u\|_{k+1}+\delta_{k,2}\|u\|_4)
\|e_0\|.
\end{equation}
Finally, by inserting (\ref{phi-psi-eh}) and (\ref{raining.400})
into (\ref{5.4}) we obtain
$$
\|e_0\|^2\leq C (h^{t_0-1} \3bar e_h\3bar +
h^{k+t_0-2}(\|u\|_{k+1}+\delta_{k,2}\|u\|_4))\|e_0\|,
$$
which, together with the estimate (\ref{4}) in Theorem
\ref{Theorem6.6}, gives rise to the desired $L^2$ error estimate
(\ref{5.3}). This completes the proof of the theorem.
\end{proof}

\medskip
The $H^2$ error estimate (\ref{4}) and the $L^2$ error estimate
(\ref{5.3}) can be used to derive some error estimates for the WG
solution $u_b$ and $\bu_g$. More precisely, observe that $e_b$ and
$\be_g$ can be represented by $e_0$ and $\partial_{ij,w}^2e_h$ by
choosing special test functions $v$ in the error equation
(\ref{4.1}). For example, $e_b$ can be represented by $e_0$ and
$\partial_{ij,w}^2e_h$ by selecting $v=\{0, v_b, 0\}$. The
representation is expressed through an equation defined locally on
each edge $e\in\E_h^0$. The rest of the analysis should be
straightforward. Details are omitted due to page limitation.

\section{Numerical Experiments}\label{Section:NE}
In this section, we present some numerical results for the WG finite
element method analyzed in previous sections. The goal is to
demonstrate the efficiency and the convergence theory established
for the method. For simplicity, we implement the lowest order scheme
for the Weak Galerkin Algorithm (\ref{2.7}). In other words, the
implementation makes use of the following finite element space
$$
\widetilde{V}_h=\{v=\{v_0,v_b,\textbf{v}_g\}, v_0\in P_2(T), v_b\in
P_0(e), \textbf{v}_g\in [P_0(e)]^2, T\in {\cal T}_h, e\in
 \E_h \}.
$$

For any given $v=\{v_0,v_b,\textbf{v}_g\}\in \widetilde{V}_h$, the
discrete weak partial derivative $\partial^2_{ij,w,r,T} v$ is
computed as a constant locally on each element $T$ by solving the
following equation
  \begin{equation*}
  \begin{split}
 (\partial^2_{ij,w,r,T}v,\varphi)_T=(v_0,\partial^2
 _{ji}\varphi)_T-\langle v_b,\partial_j\varphi\cdot n_i\rangle_{\partial T}
 +\langle v_{gi}\cdot n_j,\varphi\rangle_{\partial T},
 \end{split}
 \end{equation*}
for all $\varphi \in P_0(T)$. Since $\varphi \in P_0(T)$, the above
equation can be simplified as
  \begin{equation}
  \begin{split}
 (\partial^2_{ij,w,r,T}v,\varphi)_T=\langle v_{gi}\cdot n_j,\varphi\rangle_{\partial T},
 \qquad \forall \varphi\in P_0(T),\;
 i,j=1,2.
 \end{split}
 \end{equation}

The error for the solution of the Weak Galerkin Algorithm
(\ref{2.7}) is measured in four norms or semi-norms defined as
follows:
\begin{equation}\label{z}
\begin{split}
\3bar v\3bar^2=&\sum_{T\in {\cal
T}_h}\Bigg(\sum_{i,j=1}^d\int_T(\partial^2_{ij,w}v_h)^2dx+
h_T^{-1}\int_{\partial T} |Q_b(\nabla
v_0)-\textbf{v}_g|^2ds \\
&+ h_T^{-3}\int_{\partial T}(Q_bv_0-v_b)^2ds\Bigg),\qquad(\text{ A
discrete $H^2$-norm}),
\end{split}
\end{equation}
\begin{equation}\label{z1}
\|v\|^2=\sum_{T\in {\cal T}_h} \int_T v_0^2dx,\qquad(\text{
Element-based $L^2$-norm}),
\end{equation}
\begin{equation}\label{ubinfty}
\|v_b\|_\infty=\max_{e\in\E_h} \|v_b\|_\infty, \qquad(\text{
Edge-based $L^\infty$-norm}),
\end{equation}
\begin{equation}\label{uginfty}
\|\bv_g\|_\infty=\max_{e\in\E_h} \|\bv_g\|_\infty, \qquad(\text{
Edge-based $L^\infty$-norm}).
\end{equation}

\smallskip

The numerical experiment is conducted for the biharmonic equation
(\ref{0.1}) on the unit square domain $\Omega=(0,1)^2$. The function
$f=f(x,y)$ and the two boundary conditions are computed to match the
exact solution in each test case. The WG finite element scheme
(\ref{2.7}) was implemented on two type of partitions: (1) uniform
triangular partition, and (2) uniform rectangular partition. The
uniform rectangular partition was obtained by partitioning the
domain into $n\times n$ sub-rectangles as tensor products of 1-d
uniform partitions. The triangular meshes are constructed from the
rectangular partition by dividing each square element into two
triangles by the diagonal line with a negative slope. The mesh size
is denoted by $h=1/n$.

Table \ref{NE:TRI:Exact1} demonstrates the performance of the code
when the exact solution is given by $u=x^2+y^2+xy+x+y+1$. In theory,
the WG finite element method is exact for any quadratic polynomials.
The computational results are in consistency with theory. This table
indicates that the code should be working.

\begin{table}[H]
\begin{center}
 \caption{Numerical error for the biharmonic equation with
 exact solution $u=x^2+y^2+xy+x+y+1$ on triangular partitions.}\label{NE:TRI:Exact1}
\begin{tabular}{|c|c|c|c|c|}
\hline
$h$    & $\|u_0 -Q_0u\| $ &  $\3bar u_h -Q_hu\3bar$ & $\|u_b-Q_b u\|_{\infty}$ & $\|\bu_g-Q_b(\nabla u)\|_{\infty}$ \\
\hline
  1         &    1.73e-014  & 2.03e-014   &  1.60e-014    &  4.44e-016  \\
\hline
5.0000e-01   &  4.35e-014   &   1.88e-013  & 5.82e-014  &  6.66e-015    \\
\hline
2.5000e-01  &    1.64e-013 &  1.60e-012  &  2.86e-013   &  8.44e-014  \\
\hline
1.2500e-01  &    7.68e-013  &  7.68e-013 & 1.48e-012 & 1.19e-012  \\
\hline
6.2500e-02   &  3.65e-012  &  9.92e-011  & 7.71e-012  &  1.34e-011   \\
\hline
3.1250e-02   &  1.43e-011   &  5.20e-010   & 5.19e-010 & 3.13e-011  \\
\hline
1.5625e-02  & 4.85e-011  &   3.40e-009 & 1.15e-010  &  5.98e-010   \\
\hline
\end{tabular}
\end{center}
\end{table}

Tables \ref{NE:TRI:Case1-1} and \ref{NE:TRI:Case1-2} show the
numerical results when the exact solution is given by
$u=x^2(1-x)^2y^2(1-y)^2$. This case has a homogeneous boundary
condition for both Dirichlet and Neumann. It shows that the
convergence rates for the solution of the Weak Galerkin Algorithm in
the $H^2$ and $L^2$ norms are of order $O(h)$ and $O(h^2)$,
respectively. The numerical results are in consistency with theory
for the $L^2$ and $H^2$ norm of the error. For the approximation of
$u$ on the edge set $\E_h$, it appears that the $L^\infty$ error is
of order $\mathcal{O}(h^2)$. But the order of convergence for the
approximation of $\nabla u$ on the edge set $\E_h$ is hard to
extract from the data. It is interesting to see that the absolute
error for both $u_b$ and $\bu_g$ is quite small.

\begin{table}[H]
\begin{center}
 \caption{Numerical error and convergence order for
 exact solution $u=x^2(1-x)^2y^2(1-y)^2$ on triangular partitions.}\label{NE:TRI:Case1-1}
 \begin{tabular}{|c|c|c|c|c|}
\hline
$h$        & $\|u_0 -Q_0u\| $ & order &  $\3bar u_h -Q_hu\3bar $  & order   \\
\hline
  1         &  0.41325    &   & 0.52598 &     \\
\hline
5.0000e-01  &0.07371    &   2.49& 0.31309 & 0.75  \\
\hline
2.5000e-01 &   0.019859  &  1.89 & 0.18972 & 0.72 \\
\hline
1.2500e-01 & 0.005176&  1.94 & 0.100557 &   0.92\\
\hline
6.2500e-02  &   0.0013833 & 1.90& 0.05240 &   0.94\\
\hline
3.1250e-02  & 3.7499e-004  & 1.88&  0.02729 &   0.94 \\
\hline
1.5625e-02  & 9.977e-005   & 1.91  & 0.014058   &    0.96  \\
\hline
7.8125e-03&  2.583e-05& 1.95 &0.007145 & 0.98\\
\hline
\end{tabular}
\end{center}
\end{table}

\begin{table}[H]
\begin{center}
\caption{Numerical error and convergence order for
 exact solution $u=x^2(1-x)^2y^2(1-y)^2$ on triangular partitions.}\label{NE:TRI:Case1-2}
\begin{tabular}{|c|c|c|c|c|}
\hline
$h$        & $\|u_b-Q_b u\|_\infty$ & order & $\|\bu_g-Q_b(\nabla u)\|_{\infty}$ & order   \\
\hline
  1         &    0.41494 &  &  8.6485e-018 &     \\
\hline
5.0000e-01  &  0.08806 &   2.24&  0.00942&   \\
\hline
2.5000e-01 &  0.037013 & 1.25& 0.00491  & 0.94  \\
\hline
1.2500e-01 & 0.01069  &1.79 &   0.00354 & 0.47\\
\hline
6.2500e-02  &    0.00293 & 1.87&  0.00222 &   0.67\\
\hline
3.1250e-02  &    7.935e-004 &1.88 & 0.00102&  1.12 \\
\hline
1.5625e-02  &  2.096e-004  &  1.92 & 3.577e-004&   1.51  \\
\hline
7.8125e-03&   5.401e-05 &1.96 & 1.053e-04 & 1.76\\
\hline
\end{tabular}
\end{center}
\end{table}

Tables \ref{NE:TRI:Case2-1} and \ref{NE:TRI:Case2-2} present some
numerical results when the exact solution is given by $u=\sin(x)\sin
(y)$. We would like to invite the readers to draw conclusions from
these data.

\begin{table}[H]
\begin{center}
\caption{Numerical error and convergence order for exact solution
$u=\sin(x)\sin (y)$ on triangular partitions.}\label{NE:TRI:Case2-1}
\begin{tabular}{|c|c|c|c|c|}
\hline
$h$        & $\|u_0 -Q_0u\| $ & order &  $\3bar u_h -Q_hu\3bar$  & order   \\
\hline
  1         &   0.23000 & &  0.37336 &     \\
\hline
5.0000e-01  &  0.03575   & 2.68& 0.27641&  0.43 \\
\hline
2.5000e-01 &   0.00684  &  2.38 &   0.21911&0.34 \\
\hline
1.2500e-01 &   0.00147  &  2.21 & 0.17661& 0.31\\
\hline
6.2500e-02  &  4.427e-004 &  1.74  &  0.12349 &0.52\\
\hline
3.1250e-02  &  1.549e-004 &1.52&  0.07290 & 0.76 \\
\hline
1.5625e-02  & 4.658e-005  &  1.73  & 0.03916 &  0.90   \\
\hline
\end{tabular}
\end{center}
\end{table}

\begin{table}[H]
\begin{center}
 \caption{Numerical error and convergence order for exact solution
$u=\sin(x)\sin (y)$ on triangular partitions.}\label{NE:TRI:Case2-2}
\begin{tabular}{|c|c|c|c|c|} \hline
$h$        & $\|u_b-Q_b u\|_\infty$ & order & $\|\bu_g-Q_b(\nabla
u)\|_{\infty}$ & order   \\
 \hline
  1         & 0.21688   & & 0.06306 &     \\
\hline
5.0000e-01  & 0.05108   &  2.09 & 0.05601 &0.17  \\
\hline
2.5000e-01 &  0.01132  & 2.17& 0.05062  &0.15 \\
\hline
1.2500e-01 & 0.002524   & 2.17 & 0.03606 & 0.49\\
\hline
6.2500e-02  & 8.032e-004   &1.65 &  0.01772&  1.03\\
\hline
 3.1250e-02  &3.226e-004 & 1.32 &    0.00590& 1.59 \\
\hline
1.5625e-02  &   1.038e-004&   1.64 & 0.00163 &    1.85   \\
\hline
\end{tabular}
\end{center}
\end{table}

Table \ref{NE:TRI:case4} demonstrates the performance of the WG
finite element method when the exact solution is a biquadratic
polynomial. It shows that the $L^2$ convergence is of order
$\mathcal{O}(h^2)$, and the $H^2$ convergence has a rate
approximately $\mathcal{O}(h)$.

\begin{table}[H]
\begin{center}
\caption{Numerical error and convergence rates for the biharmonic
equation with exact solution $u=x(1-x)y(1-y)$ on triangular
meshes.}\label{NE:TRI:case4}
\begin{tabular}{|c|c|c|c|c|} \hline
$h$         &  $\|u_0-Q_0u\|$ & order  &$\3baru_h-Q_hu\3bar$   & order \\
\hline
 1         & 2.05586 & & 4.05772 & \\
\hline
5.0000e-01& 0.32234 &2.67 &  1.59961 &1.34 \\
\hline
2.5000e-01& 0.06654 &2.28 & 0.70890 &  1.17 \\
\hline
1.2500e-01& 0.01588 &2.07 & 0.34325 &1.05\\
\hline
6.2500e-02& 0.00394 & 2.01&0.17416  &  0.98 \\
\hline
3.1250e-02 &9.691e-4  & 2.02 & 0.09025  &0.95 \\
\hline
1.5625e-02 & 2.361e-4 & 2.04 &0.046632& 0.95\\
\hline
\end{tabular}
\end{center}
\end{table}

The rest of the section will present some numerical results on
rectangular meshes. The lowest order WG element on rectangles
consists of quadratic polynomials on each element enriched with
constants on the edge of each element for both $u$ and $\nabla u$.
Therefore, the total number of unknowns on each element is $18$.
Note that all the unknowns corresponding to $u$ on each element can
be eliminated locally, so that the actual number of unknowns on each
element is $12$. Table \ref{NE:REC:Exact1} shows the numerical
solution when the exact solution is a quadratic polynomial. It can
be seen that the numerical solution is numerically the same as the
exact solution, as predicted by the theory.

 \begin{table}[H]
\begin{center}
\caption{Numerical error for the biharmonic equation with
 exact solution $u=x^2+y^2+xy+x+y+1$ on rectangular partitions.}\label{NE:REC:Exact1}
\begin{tabular}{|c|c|c|c|c|}
\hline
$h$    & $\|u_0 -Q_0u\| $ &  $\3bar u_h -Q_hu\3bar$ & $\|u_b-Q_b u\|_{\infty}$ & $\|\bu_g-Q_b(\nabla u)\|_{\infty}$ \\
\hline
  1         &  2.09e-015  &    0  & 0 &    0    \\

\hline
2.5000e-01 &  4.66e-015 &  2.94e-014  & 6.66e-015  & 3.55e-015 \\
\hline

6.2500e-02 &  9.91e-014 &  2.30e-012  & 1.98e-013 &  2.74e-013 \\
\hline

1.5625e-02  & 5.83e-012 &  2.06e-010 & 1.34e-011 & 4.10e-011 \\
\hline
\end{tabular}
\end{center}
\end{table}

Tables \ref{NE:REC:Case1-1} and \ref{NE:REC:Case1-2} show the
numerical results when the exact solution is given by
$u=x^2(1-x)^2y^2(1-y)^2$. The result is in consistency with the
theory.

\begin{table}[H]
\begin{center}
 \caption{Numerical error and convergence order for
 exact solution $u=x^2(1-x)^2y^2(1-y)^2$ on rectangular partitions.}\label{NE:REC:Case1-1}
 \begin{tabular}{|c|c|c|c|c|}
\hline
$h$        & $\|u_0 -Q_0u\| $ & order &  $\3bar u_h -Q_hu \3bar$  & order   \\
\hline
  1         & 1.15052  & &   0   &     \\
\hline
5.0000e-01 &    0.14880&  2.95 &  0.35  &  \\
\hline
2.5000e-01 &  0.03786 &  1.97 &    0.24649&0.52\\
\hline
1.2500e-01 &   0.009724&  1.96 &  0.13593& 0.86\\
\hline
 6.2500e-02 &  0.002494 & 1.96  &  0.070216 & 0.95\\
\hline
3.1250e-02  &    6.509e-004 &1.94& 0.035987& 0.96 \\
\hline
1.5625e-02  &  1.709e-004&   1.93&  0.018427&   0.97  \\
\hline 7.8125e-03  &   4.415e-005 &1.95   &0.009357& 0.98  \\
\hline
\end{tabular}
\end{center}
\end{table}

\begin{table}[H]
\begin{center}
\caption{Numerical error and convergence order for
 exact solution $u=x^2(1-x)^2y^2(1-y)^2$ on rectangular partitions.}\label{NE:REC:Case1-2}
\begin{tabular}{|c|c|c|c|c|}
\hline
$h$        & $\|u_b-Q_b u\|_\infty$ & order & $\|\bu_g-Q_b(\nabla u)\|_{\infty}$ & order   \\
\hline
  1         & 0  & &    0  &     \\
\hline
5.0000e-01 &  0.15414 &   & 0.01343  &  \\
\hline
2.5000e-01 &   0.06724&   1.20&  0.008681 &0.6297\\
\hline
1.2500e-01 &  0.01961 &  1.78  & 0.0034078  &1.3490\\
\hline
 6.2500e-02 &   0.00518 &  1.92 &  0.0014578 & 1.2251\\
\hline
3.1250e-02  &  0.001359&  1.93& 8.774e-004 &   0.7325 \\
\hline
1.5625e-02  &  3.566e-004&1.93 & 3.788e-004&  1.2116 \\
 \hline
7.8125e-03  &   9.195e-005 &1.96    &1.231e-004 &1.6211   \\
\hline
\end{tabular}
\end{center}
\end{table}

Tables \ref{NE:REC:Case2-1} and \ref{NE:REC:Case2-2} present some
results for the exact solution $u=\sin(x)\sin (y)$. Readers are
encouraged to compare the results here with those in Tables
\ref{NE:TRI:Case2-1} and \ref{NE:TRI:Case2-2}.

\begin{table}[H]
\begin{center}
\caption{Numerical error and convergence order for exact solution
$u=\sin(x)\sin (y)$ on triangular partitions.}\label{NE:REC:Case2-1}
\begin{tabular}{|c|c|c|c|c|}
\hline
$h$        & $\|u_0 -Q_0u\| $ & order &  $\3bar u_h -Q_hu\3bar$  & order   \\
\hline
  1         &  0.60602  & &    0  &     \\
\hline
5.0000e-01 &  0.08424 &  2.85  & 0.26684  &  \\
\hline
2.5000e-01 &  0.01549   &   2.44 & 0.22733  &  0.23\\
\hline
1.2500e-01 &  0.00360 &  2.10  & 0.18593  &0.29\\
\hline
 6.2500e-02 &0.00101  &  1.83  & 0.13440   & 0.47\\
\hline
3.1250e-02  &  2.98e-004  &  1.77& 0.081869 & 0.72 \\
\hline 1.5625e-02  &  7.95e-005  &  1.91 & 0.044701 & 0.87 \\
 \hline
\end{tabular}
\end{center}
\end{table}

\begin{table}[H]
\begin{center}
\caption{Numerical error and convergence order for exact solution
$u=\sin(x)\sin (y)$ on rectangular
partitions.}\label{NE:REC:Case2-2}
\begin{tabular}{|c|c|c|c|c|} \hline
$h$        & $\|u_b-Q_b u\|_\infty$ & order & $\|\bu_g-Q_b(\nabla
u)\|_{\infty}$ & order   \\
\hline
  1         & 0  & &   0   &     \\
\hline
5.0000e-01 &   0.10202&   &   0.06063  &  \\
\hline
2.5000e-01 &  0.02488 & 2.04  &   0.051219  &0.24\\
\hline
1.2500e-01 &  0.006110& 2.03  &   0.039518 &0.37\\
\hline
 6.2500e-02 &  0.001981 & 1.62  & 0.021362  &0.89\\
\hline
3.1250e-02  &  5.810e-004  & 1.77&  0.007942 &  1.43 \\
\hline
1.5625e-02  &     1.501e-004 &  1.95&   0.002355&  1.75 \\
\hline
\end{tabular}
\end{center}
\end{table}

More numerical experiments should be conducted for the Weak Galerkin
Algorithm (\ref{2.7}), particularly for elements of order higher
than $k=2$. There is also a need of developing fast solution
techniques for the matrix problem arising from the WG finite element
scheme (\ref{2.7}). Numerical experiments on finite element
partitions with arbitrary polygonal element should be conducted for
a further assessment of the WG method.

\section{Appendix}\label{Section:Appendix}
The goal of this Appendix is to establish some fundamental estimates
useful in the error estimate for general weak Galerkin finite
element methods.

For any $T\in\T_h$, let $\varphi$ be a regular function in $H^1(T)$.
The following trace inequality holds true \cite{wy3655}:
\begin{equation}\label{trace-inequality}
\|\varphi\|_e^2 \leq C
(h_T^{-1}\|\varphi\|_T^2+h_T\|\nabla\varphi\|_T^2),
\end{equation}
If $\varphi$ is a polynomial on the element $T\in \T_h$, then we
have from the inverse inequality (see also \cite{wy3655}) that
\begin{equation}\label{x}
\|\varphi\|_e^2 \leq C h_T^{-1}\|\varphi\|_T^2.
\end{equation}
Here $e$ is an edge/face on the boundary of $T$.

\begin{lemma}\label{Lemma:A.01}
For the discrete weak partial derivative $\partial^2_{ij,w}$, the
following identity holds true on each element $T\in \T_h$:
\begin{equation}\label{A.001}
(\partial^2_{ij,w}v,
\varphi)_T=(\partial^2_{ij}v_0,\varphi)_T+\langle
v_0-v_b,\partial_j\varphi\cdot n_i \rangle_{\partial T}-\langle
\partial_i v_0-v_{gi},\varphi n_j\rangle_{\partial T}
\end{equation}
for all $\varphi\in P_{k-2}(T)$. Consequently, we have
\begin{equation}\label{A.002}
(\partial^2_{ij,w}v,
\varphi)_T=(\partial^2_{ij}v_0,\varphi)_T+\langle
Q_bv_0-v_b,\partial_j\varphi\cdot n_i \rangle_{\partial T}-\langle
Q_b(\partial_i v_0)-v_{gi},\varphi n_j\rangle_{\partial T}
\end{equation}
\end{lemma}

\begin{proof} From the definition (\ref{2.4}) of the weak partial derivative, we
have
\begin{equation*}
\begin{split}
(\partial^2_{ij,w}v, \varphi)_T
=&(v_0,\partial^2_{ji}\varphi)_T+\langle v_{gi}\cdot
n_j,\varphi\rangle_{\partial T}
-\langle v_b,\partial_j\varphi\cdot n_i\rangle_{\partial T}\\
=&(\partial^2_{ij}v_0,\varphi)_T-\langle \partial_i v_0,\varphi\cdot
n_j\rangle_{\partial T} +\langle v_0,\partial_j\varphi\cdot
n_i\rangle_{\partial
T}\\
&+\langle v_{gi}\cdot n_j,\varphi\rangle_{\partial T}
-\langle v_b,\partial_j\varphi\cdot n_i\rangle_{\partial T}\\
=&(\partial^2_{ij}v_0,\varphi)_T+\langle
v_0-v_b,\partial_j\varphi\cdot n_i \rangle_{\partial T}-\langle
\partial_i v_0-v_{gi},\varphi n_j\rangle_{\partial T}.
\end{split}
\end{equation*}
Here we have used the usual integration by parts in the second line.
The result then follows.
\end{proof}

\begin{lemma} Let $e_h\in V_h^0$ be any finite element function.
Then, there holds
\begin{equation}\label{ap1}
\sum_{T\in {\cal T}_h}|e_0|_{2,T}\leq C\3bar e_h \3bar,
\end{equation}
where, by definition (\ref{3barnorm}),
\begin{equation}\label{a}
\begin{split}
\3bare_h\3bar^2=&\sum_{T\in {\cal
T}_h}\sum_{i,j=1}^d(\partial^2_{ij,w}e_h,\partial^2_{ij,w}e_h)_T
+\sum_{T\in {\cal T}_h}h_T^{-3}\langle Q_be_0-e_b,Q_be_0-e_b
\rangle_{\partial T}\\
&+\sum_{T\in {\cal T}_h}h_T^{-1}\langle Q_b(\nabla
e_0)-\textbf{e}_g,Q_b(\nabla e_0)-\textbf{e}_g\rangle_{\partial T}.
\end{split}
\end{equation}
\end{lemma}
\begin{proof} Using (\ref{A.002}) with $v=e_h$ and
$\varphi=\partial^2_{ij}e_0$ we obtain
\begin{equation*}
\begin{split}
(\partial^2_{ij,w}e_h,\partial^2_{ij}e_0)_T =& (\partial^2_{ij}e_0,
\partial^2_{ij}e_0)_T-\langle Q_{b}(\partial_i e_0)-e_{gi},
\partial^2_{ij}e_0\cdot n_j\rangle_{\partial T}\\
&+\langle Q_be_0-e_b,\partial_j(\partial^2_{ij}e_0)\cdot
n_i\rangle_{\partial T}.
\end{split}
\end{equation*}
Thus,
\begin{equation}\label{x1}
\begin{split}
(\partial^2_{ij}e_0, \partial^2_{ij}e_0)_T
=&(\partial^2_{ij,w}e_h,\partial^2_{ij}e_0)_T+\langle
Q_{b}(\partial_i e_0)-e_{gi},\partial^2_{ij}e_0\cdot n_j\rangle_{\partial T}\\
&-\langle Q_be_0-e_b,\partial_j(\partial^2_{ij}e_0)\cdot
n_i\rangle_{\partial T}.
\end{split}
\end{equation}
It then follows from (\ref{x1}), Cauchy-Schwarz inequality, the
inverse inequality and  (\ref{x}) that
\begin{equation*}
\begin{split}
(\partial^2_{ij}e_0, \partial^2_{ij}e_0)_T\leq &
\|\partial^2_{ij,w}e_h\|_T\|\partial^2_{ij}e_0\|_T+\|Q_{b}(\partial_i
e_0)-e_{gi}\|_{\partial T}\|\partial^2_{ij}e_0\|_{\partial T}\\
&+\|Q_be_0-e_b\|_{\partial T}\|\partial_j(\partial^2_{ij}e_0)\|_{\partial T}\\
\leq
&\|\partial^2_{ij,w}e_h\|_T\|\partial^2_{ij}e_0\|_T+Ch_T^{-\frac{1}{2}}\|Q_{b}
(\partial_i e_0)-e_{gi}\|_{\partial T}\|\partial^2_{ij}e_0\|_{T}\\
&+Ch_T^{-\frac{3}{2}}\|Q_be_0-e_b\|_{\partial T}
\|\partial^2_{ij}e_0\|_{T},
\end{split}
\end{equation*}
which implies
\begin{equation}\label{s}
\|\partial^2_{ij}e_0\|_T^2\leq
\|\partial^2_{ij,w}e_h\|_T^2+Ch_T^{-1}\|Q_{b} (\partial_i
e_0)-e_{gi}\|_{\partial T}^2 +Ch_T^{-3}\|Q_be_0-e_b\|_{\partial
T}^2.
\end{equation}
Summing over $T\in\T_h$ completes the proof of the lemma.
\end{proof}

\begin{lemma}\label{Lemma6.2} For any $e_h\in V_h^0$ and $k\geq 3$, there exists a constant
$C$ such that
\begin{equation}\label{a1}
\Big(\sum_{T\in {\cal T}_h}h_T^{-3}\|e_0-e_b\|_{\partial
T}^2\Big)^{\frac{1}{2}}\leq C \3bare_h\3bar.
\end{equation}
\end{lemma}

\begin{proof} By the triangle inequality and the error estimate for the
projection $Q_b$, we have
\begin{equation*}
\begin{split}
h_T^{-3}\|e_0-e_b\|_{\partial T}^2
\leq & 2 h_T^{-3}\Big(\|e_0-Q_be_0\|^2_{\partial T}+\|Q_be_0-e_b\|^2_{\partial T}\Big)\\
\leq &  2h_T^{-3} \big(Ch_T^2|e_0|_{2,\partial
T}\big)^2+2h_T^{-3}\|Q_be_0-e_b\|^2_ {\partial T}\\
\leq & 2Ch_T|e_0|_{2,\partial T}^2+2h_T^{-3}\|Q_be_0-e_b\|^2_
{\partial T}\\
\leq & 2C |e_0|_{2,T}^2+2h_T^{-3}\|Q_be_0-e_b\|^2_
{\partial T}.\\
 \end{split}
\end{equation*}
Combining the above with (\ref{ap1}) gives (\ref{a1}).
\end{proof}

\begin{lemma} (Poincar\'e Inequality) There exists a constant $C$ such that
\begin{equation}\label{5.45}
\sum_{T\in {\cal T}_h}\|e_0\|^2_{T}\leq C \Big(\sum_{T\in {\cal
T}_h}\|\nabla e_0\|^2_{T}+\sum_{T\in {\cal T}_h}
h_T^{-1}\|e_0-e_b\|^2_{\partial T}\Big),
\end{equation}
where $e_h\in V_h$ is any finite element function with $e_b=0$.
\end{lemma}

\begin{proof} Consider the Laplace equation:
\begin{equation*}
\begin{array}{cc}
  -\Delta \phi =e_0 & \text{in} \ \Omega,  \\
 \phi=0 & \text{on} \ \partial\Omega.\\
\end{array}
 \end{equation*}
Assume that the solution $\phi$ is regular so that
\begin{equation}\label{aaa}
\|\phi\|^2_{2}\leq C\|e_0\|^2.
\end{equation}
The above assumption is always satisfied since otherwise we may
extend the domain $\Omega$ to $\widetilde{\Omega}$ in which the
required regularity is satisfied, with $e_0$ being extended by zero
outside of $\Omega$.

By letting $\bw=-\nabla \phi$, we have
\begin{equation*}
\begin{split}
\sum_{T\in {\cal T}_h}(e_0,e_0)_T=&\sum_{T\in {\cal
T}_h}(e_0,\nabla\cdot \bw)_T =\sum_{T\in {\cal T}_h}\langle e_0,
\bw\cdot \textbf{n}\rangle_{\partial T}-\sum_{T\in {\cal T}_h}(\bw,
\nabla e_0)_T\\
=&\sum_{T\in {\cal T}_h}\langle (e_0-e_b), \bw\cdot
\textbf{n}\rangle_{\partial T}-\sum_{T\in {\cal T}_h}(\bw,
\nabla e_0)_T\\
 \leq &\sum_{T\in {\cal T}_h}\|\bw\|_{T}\|\nabla
e_0\|_{T}+\sum_{T\in {\cal T}_h}\|\bw\|_{\partial
T}\|e_0-e_b\|_{\partial T}.
\end{split}
\end{equation*}
The trace inequality (\ref{trace-inequality}) implies
$$
\|\bw\|_{\partial T}^2\leq C(h_T^{-1}\|\bw\|_{T}+h_T\|\nabla
\bw\|_{T})\leq Ch_T^{-1}\|\bw\|_{1,T}^2.
$$
Thus, from Cauchy-Schwarz and the regularity (\ref{aaa}) we obtain
\begin{equation*}
\begin{split}
\sum_{T\in {\cal T}_h}(e_0,e_0)_T
 \leq& \sum_{T\in {\cal T}_h}\|\bw\|_{1,T}\|\nabla
e_0\|_{T}+\sum_{T\in {\cal T}_h}Ch_T^{-\frac12}\|\bw\|_{1, T}\|e_0-e_b\|_{\partial T}\\
 \leq &  C\Big(\sum_{T\in {\cal T}_h}\|\nabla
e_0\|^2_{T}+\sum_{T\in {\cal T}_h}h_T^{-1 } \|e_0-e_b\|^2_{\partial
T}\Big)^{\frac12} \|\phi\|_2\\
\leq &  C\Big(\sum_{T\in {\cal T}_h}\|\nabla e_0\|^2_{T}+\sum_{T\in
{\cal T}_h}h_T^{-1 } \|e_0-e_b\|^2_{\partial
T}\Big)^{\frac12} \|e_0\|,\\
\end{split}
\end{equation*}
which verifies the estimate (\ref{5.45}).
\end{proof}

The following is another version of the Poincar\'e inequality for
functions in $V_h^0$.

\begin{lemma}\label{Lemma6.3} There exists a constant $C$ such that
 \begin{equation}\label{5.44}
  \Big(\sum_{T\in {\cal T}_h}\|\nabla e_0\|^2_{T}\Big)^{\frac{1}{2}}\leq C\3bare_h\3bar
\end{equation}
for all $e_h\in V_h^0$.
\end{lemma}

\begin{proof} Since $e_h\in V_h^0$, then we have $\textbf{e}_g=0$.
Thus, an application of (\ref{5.45}) with $e_0$ replaced by $\nabla
e_0$ yields
\begin{equation}\label{33}
\begin{split}
\sum_{T\in {\cal T}_h}\|\nabla e_0\|^2_{T}&\leq C \Big( \sum_{T\in
{\cal T}_h} |e_0|^2_{2,T}+\sum_{T\in {\cal T}_h} h_T^{-1 }\|\nabla
e_0-\textbf{e}_g\|^2_{\partial T}\Big).
\end{split}
\end{equation}
For the second term on the right-hand side of (\ref{33}), we have
\begin{equation}\label{44}
\begin{split}
&\sum_{T\in {\cal T}_h} h_T^{-1 }\|\nabla e_0-\textbf{e}_g\|^2_{\partial T}\\
\leq& 2\sum_{T\in {\cal T}_h} h_T^{-1}\|\nabla e_0-Q_b(\nabla
e_0)\|^2_{\partial T}+2\sum_{T\in {\cal T}_h} h_T^{-1}\|Q_b(\nabla
e_0)-\textbf{e}_g\|^2_{\partial T}.
\end{split}
\end{equation}
Substituting (\ref{44}) into (\ref{33}) yields
\begin{equation*}
\sum_{T\in {\cal T}_h}\|\nabla e_0\|^2_{T} \leq C\sum_{T\in {\cal
T}_h}|e_0|^2_{2,T} + C\3bar e_h\3bar^2\leq C\3bar e_h\3bar^2,
\end{equation*}
where we have used (\ref{ap1}) in the last inequality.
\end{proof}

\begin{lemma}\label{Lemma6.4}
For quadratic element $k=2$, we  assume that the exact solution $u$
of (\ref{0.1}) is sufficiently regular such that $u\in H^4(\Omega)$.
There exists a constant $C$ such that the following inequality holds
true:
\begin{equation}\label{Fork=2}
\begin{split}
\left|\sum_{T\in {\cal
T}_h}\sum_{i,j=1}^d\langle\partial_j(\partial^2_{ij} u-{\cal
Q}_h\partial^2_{ij} u)\cdot n_i, e_0-e_b\rangle_{\partial
T}\right|\leq Ch \|u\|_4\ \3bare_h\3bar.
\end{split}
\end{equation}
\end{lemma}

\begin{proof} Since ${\cal Q}_h$ is the local $L^2$ projection onto
$P_{0}(T)$, then we have
 \begin{equation}\label{e}
\begin{split}
 &\langle\partial_j(\partial^2_{ij}u-{\cal Q}_h\partial^2_{ij} u)\cdot n_i,
 e_0-e_b\rangle_{\partial T}\\
=&\langle\partial_j\partial^2_{ij}u\cdot
n_i,e_0-e_b\rangle_{\partial T}
 \\
 =&\langle\partial_j\partial^2_{ij} u \cdot n_i, e_0-Q_be_0 \rangle_{\partial T}
 +\langle\partial_j\partial^2_{ij}u\cdot n_i,
 Q_be_0-e_b\rangle_{\partial T}\\
=&\langle(I-Q_b)\partial_j\partial^2_{ij} u \cdot n_i, e_0-Q_be_0
\rangle_{\partial T}
 +\langle\partial_j\partial^2_{ij}u\cdot n_i, Q_be_0-e_b\rangle_{\partial
 T}\\
 =& J_1 + J_2.
\end{split}
\end{equation}
For the second term $J_2$, by using Cauchy-Schwarz inequality, trace
inequality (\ref{trace-inequality}) and (\ref{a}), we have
\begin{equation}\label{f}
\begin{split}
&\left|\sum_{T\in {\cal T}_h}\sum_{i,j=1}^d\langle\partial_j\partial^2_{ij}u\cdot n_i, Q_be_0-e_b\rangle_{\partial T}\right|\\
\leq & \Big(\sum_{T\in {\cal
T}_h}\sum_{i,j=1}^dh_T^3\|\partial_j\partial^2_{ij}u\|^2_{\partial
T}\Big)^{\frac{1}{2}}
\Big(\sum_{T\in {\cal T}_h}h_T^{-3}\|Q_0e_0-e_b\|^2_{\partial T}\Big)^{\frac{1}{2}}\\
\leq & C\Big(\sum_{T\in {\cal T}_h}h_T^3\big(h_T|u|^2_{4,T}+h_T^{-1}|u|^2_{3,T}\big)\Big)^{\frac{1}{2}}\3bare_h\3bar\\
\leq & Ch\big(\|u\|_3+h\|u\|_4\big)\3bare_h\3bar.
\end{split}
\end{equation}

As to the first term $J_1$, by using Cauchy-Schwarz inequality,
trace inequality (\ref{trace-inequality}), (\ref{x}), and Lemma
\ref{Lemma6.3}, we arrive at
\begin{equation}\label{bb}
\begin{split}
&\left|\sum_{T\in {\cal
T}_h}\sum_{i,j=1}^d\langle(I-Q_b)\partial_j\partial^2_{ij}u\cdot n_i
,e_0-Q_be_0\rangle_{\partial T}\right|\\
\leq &\Big(\sum_{T\in {\cal
T}_h}\sum_{i,j=1}^d\|(I-Q_b)\partial_j\partial^2_{ij}u\|_{\partial
T}^2\Big)^{\frac{1}{2}}
\Big(\sum_{T\in {\cal T}_h} \|e_0-Q_be_0\|_{\partial T}^2\Big)^{\frac{1}{2}}\\
 \leq &   C\Big(\sum_{T\in {\cal T}_h}h_T  |  u|_{4,T}^2\Big)^{\frac{1}{2}}\Big(\sum_{T\in {\cal T}_h}h_T   |e_0|_{1,  T}^2\Big)^{\frac{1}{2}}\\
 \leq &   Ch\|u\|_4\Big(\sum_{T\in {\cal T}_h}  |e_0|^2_{1,  T}\Big)^{\frac{1}{2}}
\leq C h\|u\|_4 \3bare_h\3bar.\\
 \end{split}
\end{equation}
Combining all the above inequalities gives rise to the desired
estimate (\ref{Fork=2}).
\end{proof}

\begin{lemma}\label{Lemma6.5} There exists a constant $C$ such
that the following inequality holds true:
$$
\Big(\sum_{T\in {\cal
T}_h}\sum_{i,j=1}^dh_T^{-1}\|(\partial_ie_0-e_{gi})\cdot
n_j\|_{\partial T}^2\Big)^{\frac{1}{2}}\leq C \3bare_h\3bar
$$
for any $e_h\in V_h^0$.
\end{lemma}

\begin{proof} From the triangle inequality, we have
\begin{equation}
\begin{split}\label{a2}
&\|(\partial_ie_0-e_{gi})\cdot n_j\|_{\partial T}^2
\leq  \|\partial_ie_0-e_{gi}\|_{\partial T}^2\\
\leq & 2\Big(\|\partial_ie_0-Q_{b}(\partial_ie_0)\|^2_{\partial T}+
\|Q_{b}(\partial_ie_0)-e_{gi}\|^2_{\partial T}\Big)\\
\leq &
Ch_T|e_0|_{2,T}^2+2\|Q_{b}(\partial_ie_0)-e_{gi}\|^2_{\partial T}.
 \end{split}
\end{equation}
Thus,
\begin{equation*}
\begin{split}
\sum_{T\in {\cal
T}_h}\sum_{i,j=1}^dh_T^{-1}\|(\partial_ie_0-e_{gi})\cdot
n_j\|_{\partial T}^2&\leq C \sum_{T\in\T_h} \left(|e_0|_{2,T}^2 +
h_T^{-1}\|Q_{b}(\partial_ie_0)-e_{gi}\|^2_{\partial T}\right)\\
&\leq C\3bar e_h\3bar^2.
\end{split}
\end{equation*}
This completes the proof of the lemma.
\end{proof}


\end{document}